\documentclass{svmult-turin}

\usepackage{type1cm} 
\usepackage{enumitem,hyperref}
\usepackage{makeidx}         % allows index generation
\usepackage{graphicx}        % standard LaTeX graphics tool
\usepackage[bottom]{footmisc}% places footnotes at page bottom

\usepackage{newtxtext}      % 
\usepackage{newtxmath}      % selects Times Roman as basic font

\usepackage{tikz}
\usepackage{standalone}

\begin{document}

\title*{Radon Transform: Dual Pairs and Irreducible Representations}
\titlerunning{Radon Transform: Dual Pairs and Irreducible Representations}

\author{Giovanni S. Alberti and Francesca Bartolucci
and Filippo De Mari and Ernesto De~Vito}
\authorrunning{G. S. Alberti, F. Bartolucci, F. De Mari, and E. De Vito}
\institute{G.S. Alberti \at Department of Mathematics \& MaLGa center,
    University of Genoa, Via Dodecaneso 35, 16146 Genova, Italy,
    \email{alberti@dima.unige.it} \and  F. Bartolucci \at Department of Mathematics, ETH Zurich, Raemistrasse 101, 8092 Zurich, Switzerland, \email{francesca.bartolucci@sam.math.ethz.ch}
    \and F. De Mari \at Department of Mathematics \& MaLGa center,
   University of Genoa, Via Dodecaneso 35, 16146 Genova, Italy,
   \email{demari@dima.unige.it} \and E. De Vito \at Department of Mathematics \& MaLGa center,
    University of Genoa, Via Dodecaneso 35, 16146 Genova, Italy,
    \email{ernesto.devito@unige.it}} 

\maketitle

\abstract{We illustrate the general point of view developed  in [SIAM J.\ Math.\ Anal., 51(6), 4356--4381] that can be described as a variation of  Helgason's theory of dual $G$-homogeneous pairs $(X,\Xi)$ and which allows us  to prove  intertwining properties and inversion formulae of many existing Radon transforms. Here we analyze in detail one of the important aspects in the theory of dual pairs, namely the injectivity of the map label-to-manifold $\xi\to\hat\xi$ and  we prove that it is a necessary condition for  the irreducibility of the quasi-regular representation of $G$ on $L^2(\Xi)$. We further  explain how the theory in [SIAM J.\ Math.\ Anal., 51(6), 4356--4381] applies to the classical Radon and X-ray transforms in $\mathbb R^3$.}


\section{Introduction}
The circle of ideas and problems that may be collectively named ``Radon transform theory''
was born at least a century ago~\cite{GAradon17} but still abounds with questions and new perspectives that range from very concrete computation-oriented tasks to geometric or representation theoretic  issues. We may describe the heart of the matter by paraphrasing Gelfand~\cite{GAgel60}:

\vskip0.1truecm
 ``{\it Let $X$ be some space and in it let there be given certain manifolds which we shall suppose to be analytic and dependent analytically on parameters
$\xi_1, \dots,\xi_k$, that is 
$\{\hat\xi(\xi)=\hat\xi(\xi_1, \dots,\xi_k)\}$.
With a function $f$  on $X$ we associate its integrals over these manifolds:
\[
{\mathcal R}f(\xi)=\int\limits_{\hat\xi}f(x)\D m_\xi(x).
\]
We then ask whether it is possible to determine $f$ knowing  the integrals
 ${\mathcal R}f(\xi)$.}''

\vskip0.1truecm Among the many generalizations and theorems that may
be subsumed in this basic, yet profound, mathematical sketch, it is
certainly worth mentioning Helgason's contribution,
inspired~\cite{GAhelgason13} by work of Fritz~John's, in turn
triggered by Radon's original result~\cite{GAradon17} dating back to
1917. In particular, Helgason developed the notion of dual pairs and
double fibrations, whereby (Lie) groups and homogeneous spaces thereof
stand at center stage. His basic observation comes by inspecting
John's inversion formula for the integral transform--nowadays the
prototypical Radon transform--defined by integration over planes in
$\mathbb R^3$. The inversion takes the form
\[
  f(\mathbf
    x)=-\frac{1}{8\pi^2}\Delta_{\mathbf x}\Bigl(\int\limits_{S^2}
  {\mathcal R}f(\mathbf n,\mathbf n\cdot
 \mathbf x)\D\mathbf n\Bigr),
\]
where
$(\mathbf n,t)\mapsto {\mathcal R}f(\mathbf
  n,t)$ is the function on $S^2\times\mathbb R$ given by the integral
of $f$ over the plane
$\hat\xi_{\mathbf n,t}=\{\mathbf x\in\mathbb
R^3: \mathbf n\cdot \mathbf x=t\}$,  
$\Delta_{\mathbf x}$ is the Laplacian and
$\D\mathbf n$ is the Riemannian measure on the sphere $S^2$. This formula,
observes Helgason~\cite{GAhelgason13}, ``{\it involves two dual
  integrations, ${\mathcal R}f$ is the integral over the set of points
  in a plane and then $\D\mathbf n$, the integral over
  the set of planes through a point.}'' Furthermore, the domain $X$ on
which the functions of interest are defined (here $X=\mathbb R^3$) and
the set $\Xi$ of relevant manifolds (here the two-dimensional planes)
are homogeneous spaces of the same group $G$, namely the group of
isometries of $\mathbb R^3$, and enjoy a sort of duality, well
captured by the differential-geometric notion of incidence that was
introduced by Chern~\cite{GAchern42}.

Helgason proceeds on developing this duality in group-theoretic terms,  emphasizing a remarkable formal symmetry, according to which  the objects of interest come naturally in pairs, one living in $X$ and its twin in $\Xi$. 
Most notably, each point $\xi\in\Xi$ (the pair $(\xi_1,\xi_2)=(\mathbf n,t)$ in our basic example) labels one of the actual submanifolds $\hat\xi$ of $X$ on which the relevant integrals  are to be taken (the plane $\hat\xi(\mathbf n,t)$). Conversely, with each point $x\in X$ it is natural to associate the ``sheaf\,'' of planes passing through it. In the example above, this is precisely the set
$\check{\mathbf x}=\{\hat\xi(\mathbf
  n,\mathbf x\cdot\mathbf
  n):\mathbf n\in S^2\}$ over which the integral of
${\mathcal R}f$ is taken.

In the abstract setting developed by Helgason, the whole construction enjoys natural properties as long as
the mappings $\xi\mapsto\hat\xi$ and $x\mapsto\check x$ are both
injective, requirement that is then built into the definition of dual
pair and expressed algebraically. Note that in the above
  example, the map $(\mathbf n,t) \mapsto \hat\xi(\mathbf n,t)$ is
  two-to-one and this lack of injectivity is reflected by the fact that
  ${\mathcal R}f$  is an even function.
The central object is of course the Radon transform
\[
{\mathcal R}f(\xi)=\int\limits_{\hat\xi}f(x)\,\D m_{\xi}(x)
\]
for integrable functions on $X$, where $m_{\xi}$ is a suitable measure on $\hat\xi$.

Utilizing a variation of this framework, which is recalled in full detail below, we have addressed~\cite{GAsiam} some issues that are naturally expressed in this language. Our main contribution (see Theorem~\ref{ernestoduflomoore}) is a general result concerning the ``unitarization'' of ${\mathcal R}$ from $L^2(X,\D x)$ to  $L^2(\Xi,\D\xi)$ and the fact that the resulting unitary operator intertwines the quasi regular representations $\pi$ and $\hat\pi$ of $G$ on $L^2(X,\D x)$ and $L^2(\Xi,\D\xi)$, respectively. This unitarization really means first pre-composing the closure of ${\mathcal R}$ with a suitable pseudo-differential operator and then extending this composition to a unitary map, as is done in the existing and well-known precedecessors of Theorem~\ref{ernestoduflomoore}, such as those in~\cite{GAhelgason99} and in~\cite{GAsolmon76}. The representations $\pi$ and $\hat\pi$ of course play a central role and are assumed to be irreducible, and $\pi$ is assumed to be square integrable (see assumptions~\ref{GAass:irreducible} and \ref{GAass:irreducible1} below).
The combination of unitary extension and  intertwining leads to an interesting inversion formula for the true Radon transform, see Theorem~\ref{GAgeneralinversionformula}.

Compared to~\cite{GAsiam}, the present article adopts a slightly different, though fully compatible, formalism in the sense that we take here  the point of view that seems most natural in applications. Indeed, the space $X$ where the signals of interest are defined and the set of submanifolds of $X$ where integrals are to be taken are both in the foreground, and the group $G$ of geometric actions that one wants to consider comes next, taylored to the problem at hand. In this regard, it is important to observe that,  in principle, there are many different realizations of $X$ as homogeneous space, and  the choice of $G$ is tantamount to choosing the particular set of transformations (or symmetries) that one wants to focus on. What is of course essential is that they are plentiful enough. As for the submanifolds, we observe that in most applications one has in mind a prototypical submanifold $\hat\xi_0$.  We thus choose and fix $\hat\xi_0$, which we refer to as the {\it root} submanifold, as the image of the base point $x_0\in X$  under the  action of some closed subgroup $H$ of $G$. Thus $\hat\xi_0=H[x_0]$,  
and the other submanifolds are obtained by   exploiting the fact that $X$ is a transitive $G$-space. This entails that   $X$ is covered with all the shifted versions of $\hat\xi_0$ by means of the geometric transformations given by the elements of $G$. 
Incidentally, in this way one often achieves  families of foliations, and in most cases this leads to a natural splitting of the parameters in $\Xi$, those that label the foliation and those that select the leaf in the foliation.

Although largely inspired by the work of Helgason, our approach is different in several ways that are discussed in detail in Section~\ref{GAframe}. His construction rests not only on the strict invariance of the measures on $X$, $\Xi$ and $\hat\xi_0$ (versus relative invariance as in our construction) but on the fact that 
the correspondence $\xi\to\hat\xi$ between ``labels'' in the transitive $G$-space $\Xi$ and submanifolds 
of $X$ is assumed to be injective. In the present article we investigate this issue in detail and focus on the subgroup
$\widetilde{H}$ of $G$ that fixes $\hat\xi_0$, in principle larger than $H$. We find (Proposition~\ref{GAnovantesimo}) that the map $\xi\to\hat\xi$ is injective if and only if $\widetilde{H}=H$ and we further show in Theorem~\ref{GAthm:main} that, under reasonable assumptions on $\widetilde{H}$,  if this equality fails,  then $\hat\pi$ cannot be irreducible. This implies that in order for assumption~\ref{GAass:irreducible1} to be fulfilled, one must choose $H$ as large as possible among those subgroups of $G$ that fill out $\hat\xi_0$ by acting on $x_0$. Our theory is then illustrated with the help of two examples, 
namely the classical Radon transform  and  the X-ray transform  in $\mathbb R^3$, both analyzed with the group ${\rm SIM}(3)$ of rotations, dilations and translations. Again, this is different from Helgason's standard choice,  the isometry group ${\rm M}(3)$.

The paper is organized as follows. In Section~\ref{GAframe} we set up the context and recall the main
results of~\cite{GAsiam}. In Section~\ref{GAdpi} we present  a rather detailed analysis of the relations existing between the objects naturally arising from an arbitrary choice of $H$ and those that come from the maximal choice $\widetilde{H}$. This  leads to the main contribution  of this work, namely the fact that a gap   between  $\widetilde{H}$ and $H$ implies that the quasi regular representation $\hat\pi$ of $G$ on $L^2(\Xi)$ cannot be irreducible. Section~\ref{GAexamples} illustrates our theory with two classical examples in three-dimensional Euclidean space.

\section{The framework}\label{GAframe}
In this section we introduce the setting and the main result of
\cite{GAsiam}.

\subsection{Notation}
We briefly introduce the notation. We set
$\mathbb R^{\times}=\mathbb R\setminus\{0\}$ and $\mathbb R^{+}=(0,+\infty)$. The Euclidean norm of a vector
$v\in\mathbb R^d$ is denoted by $|v|$ and its scalar product with $w\in\mathbb R^d$ 
by $v\cdot w$. For any $p\in[1,+\infty]$ we denote
by $L^p(\mathbb R^d)$ the Banach space of functions $f\colon\mathbb R^d\rightarrow\mathbb C$ that are $p$-integrable with respect to the Lebesgue measure $\D x$
and, if $p=2$, the corresponding scalar product and norm are
$\langle\cdot,\cdot\rangle$ and $\|\cdot\|$, respectively. If $E$
  is a Borel subset of $\mathbb R^d$, $|E|$ denotes its Lebesgue measure.
The Fourier transform is denoted by $\mathcal F$ both on
$L^2(\mathbb R^d)$ and on  $L^1(\mathbb R^d)$, where it is 
defined by
\[
\mathcal F f({\omega})= \int\limits_{\mathbb R^d} f(x) \E^{-2\pi i\,
  {\omega}\cdot x } \D{x},\qquad f\in L^1(\mathbb R^d).
\] 
If $G$ is a locally compact second countable (lcsc) group, we denote
by $L^2(G,\mu_G)$ the Hilbert space of square-integrable functions with
respect to a left Haar measure $\mu_G$ on $G$. If $X$ is a lcsc transitive
$G$-space with origin $x_0$, we denote by $g[x]$ the action of $G$ on
$X$. A Borel measure $\nu$ on $X$ is relatively invariant if there
exists a positive character $\alpha$ of $G$ such that for any
measurable set $E\subseteq X$ and $g\in G$ it holds
$\nu(g[E])= \alpha(g)\nu(E)$. Furthermore, a Borel section is a
measurable map $s\colon X\to G$ satisfying $s(x)[x_0]=x$ and
$s(x_0)=e$, with $e$ the neutral element of $G$; a Borel section
always exists since $G$ is second countable
\cite[Theorem~5.11]{GAvaradarajan85}.  We denote the (real) general
linear group of size $d\times d$ by ${\rm GL}(d,\mathbb R)$.

Given two unitary
representations $\pi,\hat{\pi}$ of $G$   acting on two Hilbert spaces $\mathcal H$ and
$\hat{\mathcal H}$, respectively, a densely defined closed operator
  $T\colon\mathcal{H}\to \hat{\mathcal{H}}$ is called semi-invariant 
  with weight $\zeta$ if it satisfies  
\begin{align}\label{semiinvariant}
\hat{\pi}(g)T\pi(g)^{-1}=\zeta(g)T, \qquad g\in G,
\end{align}
where  $\zeta$ is a  character of $G$, see \cite{GAdumo76}. 

\subsection{Setting and assumptions}

The Radon transform of a signal
$f\colon X\to\mathbb C $ is defined as the integral of $f$ over a suitable family
$\{\hat{\xi}\}$ of subsets of $X$ indexed by a label
$\xi\in \Xi$. 

In this paper, we assume that the input space $X$ is a lcsc space and the signals are elements of the
Hilbert space $L^2(X,\D x)$, where $\D x$ is a given measure on $X$,
defined on the Borel $\sigma$-algebra of $X$ and finite on compact
subsets.

Following  Helgason's approach, the family
$\{\hat{\xi}\}$ is defined by first choosing a lcsc group $G$ acting on $X$ by a  continuous action
  \begin{equation}\label{eq:action1}
  (g,x) \mapsto g[x]
   \end{equation}
  in such a way that $X$ becomes a  transitive $G$-space. 
Then,  we fix an  origin $x_0\in X$, a closed subgroup $H$ of
$G$ and we define the root $\hat{\xi}_0$ of the family $\{\hat{\xi}\}$ as
  \begin{equation}
\hat{\xi}_0=H[x_0],\label{GAeq:21}
\end{equation}
 which is a closed $H$-invariant subset of $X$, by~\cite[Lemma~1.1]{GAhelgason99}.
 Denote  the set of left cosets by $\Xi=G/H$ and define for each $\xi=gH
 \in\Xi$ the closed subset of $X$
   \begin{equation}
\hat{\xi}= g[\hat{\xi}_0]= gH[x_0],\label{GAeq:31}
\end{equation}
which is independent of the choice of the
representative $g\in G$ of  $\xi\in G/H$.

In order to view the roles played by $X$ and $\Xi$ as somewhat symmetric, we introduce the 
stability subgroup of $G$ at $x_0$ 
\[
K=\{k\in G \colon k[x_0]=x_0\},
  \]
which is a closed subgroup  of $G$ such that $X$ can be identified
with $G/K$ by means of the map
\[ G/K\ni gK\mapsto g[x_0]\in X.\]
Conversely,  we regard  $\Xi$ as a transitive lcsc space with respect to 
the continuous action of $G$ given by
\begin{equation}\label{GAeq:xiaction}
g.\xi = (gg')H \qquad \xi=g'H\in \Xi.
\end{equation}
and we choose, as origin, the point $\xi_0=eH$, which makes \eqref{GAeq:21} and~\eqref{GAeq:31} consistent with each other (see Lemma~\ref{GAlemma1} below). 

With this setting, we need  the following conditions to hold true:\begin{enumerate}[label=A\arabic*),leftmargin=0.7cm]
\item\label{GAass:dx}  the measure $\D x$ is relatively $G$-invariant
  with character $\alpha$ and
   there exists a relatively
  invariant measure $\D\xi$ on $\Xi$   with character $\beta$; 
\item\label{GAass:dm_0} there exists a relatively
  $H$-invariant measure $m_0$ on $\hat{\xi}_0$ with character $\gamma$;
\item\label{GAass:sigma} there exist a Borel section $\sigma\colon\Xi\to
  G$ for the action \eqref{GAeq:xiaction} and a character $\iota$ of $G$ such that  
  \begin{equation}
\gamma\bigl(\sigma(\xi)^{-1}g\sigma(g^{-1}.\,\xi)
\bigr)=\iota(g),\qquad g\in G,\,\xi\in\Xi;\label{GAeq:4}
\end{equation}
\item\label{GAass:irreducible}   the quasi-regular representation $\pi$ of $G$
  acting on $L^2(X,{\rm d}x)$ as
  \[
\pi(g)f(x)=\alpha(g)^{-1/2}f(g^{-1}[x]),
    \]
  is irreducible and square-integrable;
  \item\label{GAass:irreducible1} the quasi-regular representation $\hat{\pi}$
    of $G$ acting on   $L^2(\Xi,{\rm d}\xi)$ as 
\[
\hat{\pi}(g)F(\xi)=\beta(g)^{-1/2}F(g^{-1}.\,\xi),
\]
    is irreducible;
\item\label{GAass:A}  there exists a non-trivial $\pi$-invariant subspace $\mathcal A\subseteq L^2(X,\D x)$ such that for all $f\in\mathcal A$ 
\begin{subequations}\label{GAassumptionclosability}
\begin{align}
& f(\sigma(\xi)[\cdot])\in L^1(\hat{\xi}_0,m_0)  \quad 
  \text{for almost all }\xi\in\Xi, \label{GAeq:closability1} \\
& \mathcal R f:=\int\limits_{\hat{\xi}_0}  f(\sigma(\cdot)[x])  {\rm d}m_0(x)
 \in L^2(\Xi, {\rm d}\xi),  \label{GAeq:closability2}
\end{align}
\end{subequations}
and the adjoint of the operator $\mathcal R\colon\mathcal A\to L^2(\Xi,\D\xi)$ has non-trivial domain.
\end{enumerate}

We add a few comments. The assumption that the measure $\D x$ is
(relatively) invariant ensures that the group $G$ acts also on the
signals by means of the unitary representation $\pi$.

It is worth observing that \eqref{GAeq:closability1} is independent of the section $\sigma$. Indeed, if $\sigma'$ is another section, by assumption \ref{GAass:dm_0} we have
\begin{equation*}
\begin{split}
\int\limits_{\hat{\xi}_0}|f(\sigma'(\xi)[x])|{\rm
  d}m_0(x)
  &=\int\limits_{\hat{\xi}_0}|f(\sigma(\xi)\sigma(\xi)^{-1}\sigma'(\xi)[x])|{\rm
  d}m_0(x)\\
   &=\gamma\left(\sigma'(\xi)^{-1}\sigma(\xi)\right)\int\limits_{\hat{\xi}_0}|f(\sigma(\xi)[x])|{\rm
  d}m_0(x),
    \end{split}
\end{equation*}
since $\sigma(\xi)^{-1}\sigma'(\xi)\in H$.

By means of the section $\sigma$,  the family  $\{\hat{\xi}\}$  is given by
\begin{equation}\label{spreads}
\hat{\xi}=\sigma(\xi)[\hat{\xi}_0]\subseteq X,
\end{equation}
and  the map $x\mapsto
\sigma(\xi)[x]$ is a Borel bijection from $\hat{\xi}_0$ onto
$\hat{\xi}$,  so that~\eqref{GAeq:closability2}  reads as 
\begin{equation} \label{GAeq:radondualpairs}
\mathcal{R}f(\xi)=\int\limits_{\hat{\xi}}\,f(x) {\rm d}m_\xi(x),
\end{equation}
where $m_\xi$ is the image measure of $m_0$ under the above bijection.
Hence   for any signal
belonging to $\mathcal A$, the map $\mathcal R f$ is precisely the Radon transform
of $f$.  Note that  $\mathcal A$ is a dense subspace of $L^2(X,\D x)$ by
irreducibility of $\pi$ and this property also guarantees that
the adjoint of $\mathcal R$ is uniquely defined.

Given the space of signals $L^2(X,\D x)$, there are possibly many
different pairs $(G,H)$ that give rise to the same family
$\{\hat{\xi}\}$ of subsets and (essentially)  to the same  Radon trasform $\mathcal R$.  In this
paper, $G$ is chosen in such a way that $\pi$ is a  square-integrable
representation, so that there exists a self-adjoint operator
  \[C\colon\operatorname{dom}{C}\subseteq L^2(X,{\rm d}x) \to L^2(X,{\rm d}x), \]
semi-invariant with weight $\Delta^{\frac{1}{2}}$, where
  $\Delta$ is the modular function of $G$.  Hence, for all $\psi\in
\operatorname{dom}{C}$ with $\|C\psi\|=1$, the voice transform
$\mathcal V_\psi$ 
\[
  (\mathcal V_\psi f)(g)= \langle f, \pi(g)\psi\rangle, \qquad g\in G,
 \]
 is an isometry from $L^2(X,{\rm d}x)$ into $L^2(G,\mu_G)$. In this
 case the vector $\psi$ is called admissible  
 and  we have the weakly-convergent reproducing formula 
\begin{align}\label{recon}
f=\int\limits_G   (\mathcal V_\psi f)(g)   \pi(g)\psi\ {\rm d}\mu_G(g),
\end{align}
see, for example, \cite[Theorem~2.25]{GAfuhr05}). Eq.~\eqref{recon} is at the basis of
our reconstruction formula~\eqref{GAinversionformula}. 

We stress that in Helgason's approach, the representation $\pi$ is not directly 
considered, and hence there is no need to require it to be either irreducible or square-integrable.  This 
entails  a larger freedom in the choice of the group $G$.

We  recall that, since $X$, $\Xi$ and
$\hat{\xi}_0$ are transitive spaces, there always exist quasi-invariant
measures on these three spaces. In Helgason's approach, it is assumed
that the measures are 
invariant, so that they are unique up to a constant. In this paper, we
only require that $\D x$, $\D\xi$ and $m_0$ are relatively invariant.
In particular, $m_0$ is not uniquely  given (up to a constant) and  the
definition of the Radon transform  depends not only  on the family $\{\hat\xi\}$, but
also on the measure $m_0$ and the section $\sigma$.
Since $m_0$ is not invariant, Assumption \ref{GAass:sigma}  is needed to ensure  the
right covariance properties of the Radon transform and in many
examples it can be easily satisfied by a suitable choice of the section
$\sigma$.

\subsection{The Unitarization Theorem and Inversion Formula}
The isometric extension problem for the Radon transform was actually 
addressed and implicitly solved by Helgason in the general context of symmetric
spaces, see \cite[Corollary 3.11]{GAhelgason94}. However, 
as a consequence of the intertwining properties
of the Radon transform it is possible to provide an alternative proof
of the following result,  see \cite{GAsiam}. 

\begin{theorem}\label{ernestoduflomoore}
Under the above assumptions,  
  \begin{enumerate}
  \item the Radon transform
    $\mathcal{R}\colon\mathcal A\to L^2(\Xi,{\rm d}\xi)$ admits a
    unique closure $\overline{\mathcal{R}}$;
 \item the  closure $\overline{\mathcal{R}}$ satisfies 
    \begin{equation}
      \label{character}
      \overline{\mathcal{R}}\pi(g) =\chi(g)^{-1}\hat{\pi}(g) \overline{\mathcal{R}}, 
    \end{equation}
    for all $g\in G$, where $\chi$ is the character given by
    \begin{equation}\label{chi}
      \chi(g)=\alpha(g)^{1/2}\beta(g)^{-1/2}\gamma(g\sigma(g^{-1}.\xi_0))^{-1} ;
    \end{equation}
\item there exists a unique positive self-adjoint operator
    \[ \mathcal{I}\colon \operatorname{dom}(\mathcal{I}) \supseteq
      \operatorname{Im}\overline{\mathcal{R}}\to L^2(\Xi,{\rm d}\xi),
    \]
    semi-invariant with weight $\zeta=\chi^{-1}$ with the property
    that the composite operator $\mathcal I \overline{\mathcal{R}}$ extends
    to a unitary operator
    $\mathcal Q\colon L^2(X,{\rm d}x)\to L^2(\Xi,{\rm d}\xi)$ intertwining
    $\pi$ and $\hat{\pi}$, namely
    \begin{equation}\label{intertwiningU}
      \hat{\pi}(g)\mathcal Q\pi(g)^{-1}=\mathcal Q,
      \qquad g\in G.
    \end{equation}
  \end{enumerate}
\end{theorem}
It follows that the representations  $\pi$ and $\hat{\pi}$ are
equivalent, so that $\hat{\pi}$ is square-integrable, too.

Since $\mathcal Q$ is unitary and satisfies \eqref{intertwiningU} and   $\pi$
is square-integrable, it is possible to prove the following inversion
formula for the Radon transform, \cite{GAsiam}.
\begin{theorem}\label{GAgeneralinversionformula}
 Let $\psi\in
L^2(X,{\rm d}x)$ be an admissible vector for the representation $\pi$ such that
$\mathcal Q\psi\in\operatorname{dom}\mathcal I$, and set $\Psi= \mathcal I\mathcal{Q}\psi$.
Then, for any $f\in \operatorname{dom}\overline{\mathcal{R}}$,
 \begin{equation}
    \label{GAinversionformula}
  f =  \int\limits_G \chi(g)
  \langle \overline{\mathcal{R}} f,\hat{\pi}(g)\Psi\rangle \,     \pi(g)\psi\ {\rm d}\mu_G(g),
  \end{equation}
  where the integral is weakly convergent,and 
  \begin{equation}\label{reconstructionenergy}
  \|f\|^2=\int\limits_G\chi(g)^2|\langle\overline{\mathcal{R}}f,\hat{\pi}(g)\Psi\rangle|^2{\rm d}\mu(g).
  \end{equation}
  If, in addition,  $\psi\in \operatorname{dom}\overline{\mathcal{R}}$, then 
$
\Psi=\mathcal I^2\overline{\mathcal{R}}\psi. 
$
  \end{theorem}
Note that the datum  $\overline{\mathcal{R}}f$ is
  analyzed by the family   $\{ \hat{\pi}(g)\Psi\}_{g\in G}$ and the
  signal $f$ is  reconstructed by a different family, namely $\{ \pi(g)\psi\}_{g\in G}$.
The idea to exploit the theory of the continuous wavelet transform to derive inversion formulae for the Radon transform is not new, we refer to \cite{GAberenstein-walnut-1994, GAholschneider91,GAmadych-1999,GAolson-destefano-1994,GArubin98,GAwalnut-1993,GAyarman07}--to name a few.

\section{Dual pairs and irreducibility}\label{GAdpi}

In this section, we show the relation between our setting and
the notion of dual pairs introduced by Helgason \cite{GAhelgason99} and the connection
with the assumption on the irreducibility. If we identify $X$
and $\Xi$ with the corresponding homogenous spaces $G/K$ and $G/H$, so
that $\xi=g_1 H$ for some $g_1\in G$, then a point $x=g_2K$ belongs
to $\hat{\xi}$ if and only if $g_2K\cap g_1H\neq\emptyset$, which is the
notion of incidence introduced by Chern  in~\cite{GAchern42} and adopted by Helgason.

Interchanging the roles of $X$ and $\Xi$, we can define 
\[
\check{x}_0=K.\,\xi_0\subseteq\Xi,\qquad \check{x}=s(x).\,\check{x}_0\subseteq \Xi,
\]
where $s\colon X\to G$ is any  section for the action \eqref{eq:action1}. The notion of incidence
makes clear the following duality relation
\[ x\in \hat{\xi} \qquad \Longleftrightarrow \qquad \xi\in \check{x}.\]

Furthermore, if $\check{x}_0$ admits a relatively invariant measure $\hat{m}_0$, we
can define the back-projection of a function $\hat{f}\colon\Xi\to \mathbb C$
as
\[
\mathcal R^\#\hat f(x)= \int\limits_{\check{x}_0} \hat{f}(s(x).\xi) \D \hat{m}_0(\xi) =:
\int\limits_{\check{x}} \hat{f}(\xi) \D \hat{m}_x (\xi),
  \]
provided that the integral converges, where $\hat{m}_x$ is the image measure of $m_0$ under the bijection $\xi\mapsto
s(x).\xi$ from $\check{x}_0$ onto $\check{x}$. 

The pair $(X,\Xi)$ is called a dual pair by Helgason if
both the map $\xi\mapsto \hat{\xi}$ and the map $x\mapsto \check{x}$ are
injective.  Below we provide an alternative characterization of injectivity. 
We need some preliminary facts.
\begin{lemma}\label{GAlemma1}
  For  all $g\in G$ and $\xi\in\Xi$
  \[
    g[\hat{\xi}]=\widehat{g.\xi} 
    \]
\end{lemma}
\begin{proof}
 For $g\in G$ and $\xi=g'H\in\Xi$, by equations \eqref{GAeq:31} and \eqref{GAeq:xiaction} it holds that 
\[
g[\hat{\xi}]=gg'[\hat{\xi}_0]=\widehat{g.\xi}.
\]
This concludes the proof. 
\end{proof}

  \begin{lemma}\label{GAtildeH}    
  The set 
  \[
    \widetilde{H}=\{g\in G\mid g[\hat{\xi}_0]=\hat{\xi_0}  \}  \]
  is a closed subgroup of $G$ and $\widetilde{H}\supseteq H$.
\end{lemma}
\begin{proof}
Clearly, $\widetilde{H}$ is a subgroup of $G$ and $\widetilde{H}\supseteq H$. We prove that it is
closed. Let $(g_n)_n$ be a sequence of $\widetilde{H}$ converging to $g$ and
$x\in\hat{\xi}_0$, then $(g_n[x])_n$ is a sequence of $\hat{\xi_0} $
converging to $g[x]\in \hat{\xi}_0$ since  the action is
continuous and $\hat{\xi}_0$ is closed. Then,
$g[\hat{\xi}_0]\subseteq\hat{\xi}_0$. The same argument
  applied to the sequence 
$(g_n^{-1})_n$ in $\widetilde{H}$, which converges to $g^{-1}$,
yields $g^{-1}[\hat{\xi}_0]\subseteq\hat{\xi}_0$, namely, $\hat{\xi}_0\subseteq g[\hat{\xi}_0]$.
\end{proof}
The next proposition provides an alternative characterization of the
injectivity in terms of $\widetilde{H}$. More precisely, the map $\xi\mapsto\hat\xi$ is injective if and only if $H$ is chosen as the maximal subgroups fixing $\hat\xi_0$. The reader is referred to Section~\ref{GAexamples} below for two examples of this aspect.

\begin{proposition}\label{GAnovantesimo}
The map $\xi\mapsto \hat{\xi}$ is injective if and only if $\widetilde{H}=H$.  
\end{proposition}
\begin{proof}
  Given $\xi,\xi'\in\Xi$, the condition $\hat{\xi}=\hat{\xi'}$ is
  equivalent to
  $\sigma(\xi')^{-1}\sigma(\xi)[\hat{\xi}_0]=\hat{\xi}_0$, {\textit
    i.e.},  $\sigma(\xi')^{-1}\sigma(\xi)\in \widetilde{H}$.

On the other hand, since $\xi=\sigma(\xi).\xi_0$ and
$\xi'=\sigma(\xi').\xi_0$, the condition $\xi=\xi'$ is equivalent to
$\sigma(\xi).\xi_0=\sigma(\xi').\xi_0$, {\textit
  i.e.},  $\sigma(\xi')^{-1}\sigma(\xi)\in H$.

If $H=\widetilde{H}$, it follows that $\hat{\xi}=\hat{\xi'}$ if and
only  if $\xi=\xi'$. If $H\neq\widetilde{H}$, then since $H$ is always
contained in $\widetilde{H}$, there exists $g\in \widetilde{H}\setminus H$. Then, by Lemma~\ref{GAlemma1},
\[
\hat{\xi}_0 =g[\hat{\xi}_0]=\widehat{g.\xi_0}.
  \]
However, $g.\xi_0\neq \xi_0$ because $g\notin H$. 
\end{proof}

Since $\widetilde{H}$ is closed, we can consider  the transitive space
$\widetilde{\Xi}=G/\widetilde{H}$ and, since $H$ is a closed subgroup of
$\widetilde{H}$,  the map
\[
j\colon\Xi\to\widetilde\Xi,\qquad j(gH)=g\widetilde{H},
\]
is a continuous surjection intertwining the actions of
$G$ on $\Xi$ and $\widetilde{\Xi}$,  {\em i.e.}
\[
  j(g.\xi)=g.j(\xi),\qquad g\in G,\xi\in \Xi, 
  \]
where the action of $\widetilde{\Xi}$ is still denoted by $g.\widetilde{\xi}$.
Furthermore,  
for all $\widetilde{\xi}=g\widetilde{H}\in\widetilde{\Xi}$, we define
  \[
 \hat{\widetilde{\xi}}=g\widetilde{H}[x_0].
\]
\begin{corollary}\label{GAcor:forzajuve}
For all $\xi\in \Xi$
\begin{equation}
  \label{GAGAeq:1}
  \widehat{j(\xi)} =\hat{\xi}\qquad  
\end{equation}
and the map $\widetilde{\xi}\mapsto \hat{\widetilde{\xi}}$ is injective.
\end{corollary}
\begin{proof}
Fix $\xi=gH\in\Xi$, then
 \[
\widehat{j(\xi)} 
 =g\widetilde{H}[x_0]=g(\widetilde{H}H)[x_0]=
   g\widetilde{H}[\hat{\xi}_0]= g [\hat{\xi}_0]=\hat{\xi},
   \]
where the second equality holds true since
$\widetilde{H}H=\widetilde{H}$ 
whereas the fourth equality is due to the definition of
$\widetilde{H}$.  If  $\widetilde{\xi_0}=j(\xi_0)=\widetilde{H}$ is
the origin of $\widetilde{\Xi}$, from~\eqref{GAGAeq:1}, it follows that
$\widehat{\widetilde{\xi_0}}=\hat{\xi}_0$, so that  
 $\widetilde{\widetilde{H}}=\widetilde{H}$ and the
  injectivity follows from Prop.~\ref{GAnovantesimo} with $H$ replaced by~$\widetilde{H}$. 
\end{proof}

\subsection{Irreducibility}\label{GAirriducibili}
In this section we show that if $H$ is a proper subgroup
  of  $\widetilde{H}$, then
$\hat{\pi}$ is not irreducible. To prove the claim we need that
$\widetilde{H}$ satisfies the same assumptions made on $H$ and that there is
the appropriate compatibility between the two subgroups. 

As in~\ref{GAass:dx}, we first suppose that $G/\widetilde{H}$ has a $G$-relatively invariant
measure $\D\widetilde{\xi}$ with the same character $\beta$ of
$\D\xi$.  Since  $\beta$ satisfies
    \begin{equation}
 \beta(h) = \frac{\Delta_{\widetilde{H}}(h) }{\Delta_G(h)}\quad
 h\in\widetilde{H} ,\quad 
\beta(h) = \frac{\Delta_H (h)}{\Delta_G(h)}\quad
 h\in H,     \label{GAeq:3}
\end{equation}
then
\begin{equation}
\Delta_H(h)=\Delta_{\widetilde{H}}(h),\qquad  h\in H.\label{GAeq:2}
\end{equation}
Eq.~\eqref{GAeq:2} implies that  there exists an invariant measure
$\omega$ on
$\widetilde{H}/H$,  see   \cite[Corollary~2 Section~2, No.~6
INT~VII.43]{GAbourbaki}.

Note that  $\hat{\xi_0}$ is a
transitive space with respect to the action of  $\widetilde{H}$. As in
Ass.~\ref{GAass:dm_0}, we also  assume that the measure  $m_0$ is relatively
$\widetilde{H}$-invariant with character 
$\widetilde{\gamma}$. Note that  
\[ \widetilde{\gamma}(h)=\gamma(h),\qquad h\in H.\]
Finally, strengthening the analogous of~\ref{GAass:sigma} for $\widetilde\Xi$, we suppose that $\sigma$ is such
that 
\begin{equation}
\widetilde{\gamma}( \sigma(\xi')^{-1}\sigma(\xi))=1,\label{GAeq:5}
\end{equation}
for every $\xi,\xi'\in\Xi$ with $j(\xi)=j(\xi')$.

The main result of this work reads as follows.
\begin{theorem}\label{GAthm:main}
Under the above assumptions, if $\hat\pi$ is irreducible, then $H=\widetilde H$ and the map $\xi\mapsto\hat\xi$ is injective.
\end{theorem}
The rest of this section is devoted to the proof of this result.

To prove our main result we recall the following
disintegration formula. We adopt the notation of  \cite[Definition~1,
Section~2, No.~2 INTVII.31]{GAbourbaki}. Given a character
$\beta$ of $G$ and a closed subgroup $G_0$ of $G$, we denote by
$\beta\cdot \mu_G/\mu_{G_0}$ the unique measure on the quotient space
$G/G_0$ such that for all compactly supported continuous functions $f\colon G\to\mathbb C$
  \[
\int\limits_G \beta(g) f(g) \D\mu_G(g)= \int\limits_{G/G_0} \left(\int\limits_{G_0}
f(gh) \D\mu_{G_0}(h) \right) \D \left(\beta\cdot
\mu_G/\mu_{G_0}\right)(gG_0) .
    \]
Observe first that, according to 
\cite[Theorem~3, Section~2, No.~6 INT~VII.43]{GAbourbaki}, the relatively invariant measures
$\D\widetilde{\xi}$ and $\D\xi$ are proportional to
$(\beta\cdot \mu_G)/\mu_{\widetilde{H}}$ and
$(\beta\cdot \mu_G)/\mu_{H}$, respectively.
Furthermore, note that the map
\[
(G,\widetilde{H}) \ni (g,h) \to gh \in G
\]
defines a continuous and proper right action of $\widetilde{H}$ onto $G$. The
measure $\beta\cdot \mu_G$ is right relatively $\widetilde{H}$-invariant with
character $\Delta_{\widetilde{H}}$.  Then by  \cite[Corollary~1, Section~2, No.~8 INT~VII.45]{GAbourbaki}, there exists a positive constant $C>0$ such that, for any
$f\in L^1(\Xi,d\xi)$ 
\begin{align}
  \int\limits_{\Xi} f(\xi) \D\xi  & = C\int\limits_{\widetilde{\Xi}} \left(\,\,
    \int\limits_{\widetilde{H}/H} 
    f(\widetilde{\sigma}(\widetilde{\xi})h.\xi_0) \D\omega(hH) \right) 
  \D\widetilde{\xi} \label{GAeq:6a}
\end{align}
where $\widetilde{\sigma}\colon\widetilde{\Xi}\to G$
is a  section and  the value
$f(\widetilde{\sigma}(\widetilde{\xi})h.\xi_0)$  depends only on the
left coset $hH$ since $H$ is the stability subgroup at 
$\xi_0$. The right hand side is well defined since
there is a negligible set $\widetilde{E}\subset\widetilde{\Xi}$ such 
that if $\widetilde{\xi}\not\in \widetilde{E}$,  the map
\[\widetilde{H}/H\ni hH\to
  f(\widetilde{\sigma}(\widetilde{\xi})h.\xi_0)\in \mathbb C\]
is integrable with respect to $\omega$, 
and  the almost everywhere defined function
  \[ \widetilde{\Xi}\ni\widetilde{\xi}\mapsto
    \int\limits_{\widetilde{H}/{H}}f(\widetilde{\sigma}(\widetilde{\xi})h.\xi_0) \in\mathbb C\]
is integrable with respect to $\D\widetilde{\xi}$.
Furthermore, ~\eqref{GAeq:6a} is equivalent to
  \begin{align}
  \int\limits_{\Xi} f(\xi) \D\xi   
  & = C \int\limits_{\widetilde{\Xi}} \left(\,\,
    \int\limits_{j(\xi)=\widetilde{\xi}} 
    f(\xi) d\nu_{\widetilde{\xi}}(\xi) \right) 
  \D\widetilde{\xi} \label{GAeq:6b}
  \end{align}
where $\nu_{\widetilde{\xi}}$
is the image measure of  $\omega$ under the map
\[ \widetilde{H}/H\ni hH\mapsto
\widetilde{\sigma}(\widetilde{\xi})h.\xi_0\in \Xi,\]
which is a homeomorphism from $\widetilde{H}/H$ onto the closed subset
$j^{-1}(\widetilde{\xi})$. In particular, it holds true that the support of each
$\nu_{\widetilde{\xi}}$ is $j^{-1}(\widetilde{\xi})$.  As a
consequence,  a subset $\widetilde{E}\subset\widetilde{\Xi}$ is
negligible with respect to $\D\widetilde{\xi}$ if and only if
$j^{-1}(\widetilde{E})$ is  negligible with respect to $\D\xi$.

The next  lemma shows that the Radon transform $\mathcal R f (\xi)$  depends only
on $j(\xi)$. 

\begin{lemma}\label{GAlem:radon-j} 
 For all $f\in \mathcal A$, there exists  $\widetilde{F}\colon\widetilde{\Xi}\to \mathbb C$
 and a negligible  set $\widetilde E\subseteq \widetilde \Xi$ such that $j^{-1}(\widetilde E)$ is negligible and 
\begin{equation}
  \mathcal R f(\xi)= \widetilde{F}(j(\xi)), \qquad  \xi\notin j^{-1}(\widetilde E).\label{GAeq:1}
\end{equation}
Furthermore, for every $\xi_1,\xi_2\in\Xi$
\[
\hat\xi_1=\hat\xi_2\quad\implies\quad \mathcal R f(\xi_1) = \mathcal R f(\xi_2).
\]
\end{lemma}
\begin{proof}
Given $f\in\mathcal A$,  define
  \[
    E = \{ \xi \in \Xi \colon  f(\sigma(\xi))[\cdot])\notin L^1(\hat{\xi}_0,m_0)\}.
\]
By~\eqref{GAeq:closability1}, the set $E$ is negligible. For any two points
$\xi,\xi'\in \Xi$ such that 
  $j(\xi)=j(\xi')$, taking into account that  $\sigma(\xi')^{-1}
  \sigma(\xi)\in \widetilde H$  and by assumption \eqref{GAeq:5},  it holds that
  \begin{align*}
    \mathcal R f(\xi') & = \int\limits_{\hat{\xi_0}} f(\sigma(\xi) \sigma(\xi)^{-1}
                  \sigma(\xi')[x])\,\D m_0(x) \\ & = \widetilde{\gamma}( \sigma(\xi')^{-1}
                                                 \sigma(\xi))
                                                   \int\limits_{\hat{\xi_0}}
                                                   f(\sigma(\xi))[x])\,\D
                                                   m_0(x)
    \\  & = \widetilde{\gamma}( \sigma(\xi')^{-1} \sigma(\xi)) \mathcal R f(\xi) \\&=\mathcal R f(\xi),
  \end{align*}
  so that either $\xi,\xi'\notin E$ or $\xi,\xi'\in E$ and the
  claim follows with $\widetilde E=j(E)$ because
  $E=j^{-1}(j(E))$. Since $E$ is negligible, so is
    $\widetilde{E}$ as a consequence of~\eqref{GAeq:6b}, as already observed.

  The last part immediately follows from Corollary~\ref{GAcor:forzajuve}.
\end{proof}
As shown by the following corollary, $F$ is the Radon transform of $f$
associated to the pair $(X,\widetilde{\Xi})$, which 
trivially satisfies~\ref{GAass:dx} to~\ref{GAass:irreducible}, whereas the definition of the Radon trasform requires only \eqref{GAeq:closability1}. 

\begin{corollary}
Let
  $q\colon \widetilde{\Xi} \to \Xi$ be a measurable section such that $q(\widetilde{\xi_0})=\xi_0$. Then $\widetilde\sigma=\sigma\circ q\colon\widetilde\Xi\to G$ is a measurable section.   Furthermore, the pair $(X,\widetilde{\Xi})$
  satisfies \ref{GAass:sigma} and  \eqref{GAeq:closability1} for all $f\in\mathcal A$ and, denoting  the corresponding Radon transform by
  $\widetilde{\mathcal R}$, for all $f\in\mathcal A$,  
\[ \mathcal R f(\xi)= \widetilde{\mathcal R}f(j(\xi))\qquad \text{a.e. }\xi\in\Xi,\]
and
\[
 \mathcal R f(q(\widetilde\xi))= \widetilde{\mathcal R}f(\widetilde\xi)\qquad \text{a.e. }\widetilde\xi\in\widetilde\Xi.
\]
\end{corollary}  
\begin{proof}
  Let $p\colon G\to \Xi$ and $\widetilde{p}\colon G\to\widetilde{\Xi}$ be the
  canonical projections, then $\widetilde{p}=j\circ p$.  We readily derive
  \[
    \widetilde{p} \circ \widetilde{\sigma} = j \circ
    (p\circ\sigma)\circ q = j\circ q = \operatorname{Id},\qquad
    \widetilde{\sigma}(j(\xi_0) )= \sigma(\xi_0)=e,
  \]
  so that $\widetilde{\sigma}$ is a measurable section from
  $\widetilde{\Xi}$ to $G$. 
    
   From~\eqref{GAeq:4}, with
  $\xi=q(\widetilde{\xi})$ and $g\in G$ we get
  \[
  \begin{split}
    \iota(g) & =\gamma\bigl(\sigma(q(\widetilde{\xi}))^{-1}g\sigma(g^{-1}. q(\widetilde{\xi}))
               \bigr) \\&
               = 
               \widetilde{\gamma}\bigl(\widetilde{\sigma}(\widetilde{\xi})^{-1}g\left(\widetilde{\sigma}(g^{-1}.\,
               \widetilde{\xi}) \sigma(\xi')^{-1}\right)\sigma(\xi'')\bigr)\\ & =\widetilde{\gamma}\bigl(\widetilde{\sigma}(\widetilde{\xi})^{-1}g\widetilde{\sigma}(g^{-1}.\,\widetilde{\xi}) \big) \widetilde{\gamma}( \sigma(\xi')^{-1}\sigma(\xi''))
 \\&
 =\widetilde{\gamma}\bigl(\widetilde{\sigma}(\widetilde{\xi})^{-1}g\widetilde{\sigma}(g^{-1}.\, \widetilde{\xi}) \big)            
 \end{split}
 \]
  where $\xi'= q(g^{-1}.\widetilde{\xi})$ and
  $\xi''= g^{-1}.q(\widetilde{\xi})$ are such that $j(\xi')=j(\xi'')$ so
  that the last equality is a consequence
  of~\eqref{GAeq:5}. Hence,~Assumption~\ref{GAass:sigma} holds true. 

We now prove \eqref{GAeq:closability1} for $\widetilde\sigma$. Take $f\in\mathcal A$. Let $\widetilde E\subseteq\widetilde\Xi$ be the negligible set given by Lemma~\ref{GAlem:radon-j}.  Note that if $\widetilde\xi\notin \widetilde E$ then $q(\widetilde\xi)\notin j^{-1}(\widetilde E)$. Thus, for $\widetilde\xi\notin \widetilde E$ we have
\[
\int\limits_{\hat{\xi}_0}  |f(\widetilde\sigma(\widetilde\xi))[x])|  {\rm d}m_0(x)=\int\limits_{\hat{\xi}_0}  |f(\sigma(q(\widetilde\xi))[x])|  {\rm d}m_0(x)<+\infty,
\]
and so $ f(\widetilde\sigma(\widetilde\xi)[\cdot])\in L^1(\hat{\xi}_0,m_0)$. Similarly, for $\widetilde\xi\notin \widetilde E$ we have
\[
\widetilde{\mathcal R} f(\widetilde\xi):=\int\limits_{\hat{\widetilde\xi}_0}  f(\widetilde\sigma(\widetilde\xi)[x])  {\rm d}m_0(x)=\int\limits_{\hat{\xi}_0}  f(\sigma(q(\widetilde\xi))[x]) {\rm d}m_0(x)=\mathcal R f(q(\widetilde\xi)).
\]
Finally, for $\widetilde\xi=j(\xi)$ with $\xi\notin j^{-1}(\widetilde E)$ we have $j(q(\widetilde \xi))=\widetilde\xi$, and so Lemma~\ref{GAlem:radon-j}  yields
\[
\widetilde{\mathcal R} f(j(\xi))=\mathcal R f(\xi),
\]
as desired.
\end{proof}

\begin{lemma}\label{GAlem5}
  The space
  \[
L^2(\Xi,\D\xi)_0=\{ F\in L^2(\Xi,\D\xi) \mid
  F(\xi)=\widetilde{F}(j(\xi)) \textrm{ for a.e. }  \xi\in\Xi\  \text{ for some
  }\widetilde{F}\colon\widetilde{\Xi}\to\mathbb C \}
\]
is a closed $\hat{\pi}$-invariant subspace. Hence, if
$\hat{\pi}$ is irreducible, then
\begin{equation}
  \label{GAGAeq:2}
  L^2(\Xi,\D\xi)_0= L^2(\Xi,\D\xi).
\end{equation}
\end{lemma}
\begin{proof}
  We first observe that, given $F\in L^2(\Xi,\D\xi)_0$, by
  construction  there exists $\widetilde{F}\colon\widetilde{\Xi}\to\mathbb C $ such
  that $F$ and $\widetilde{F}\circ j$ are equal almost everywhere. Hence, we can
  always assume that $F=\widetilde{F}\circ  j$. 
  
 Let $(F_n)$ be a sequence in  $L^2(\Xi,\D\xi)_0$ converging to $F\in
 L^2(\Xi,\D\xi)$. As observed, we can assume that $F_n=\widetilde{F}_n\circ  j$
where $\widetilde{F}_n\colon \widetilde{\Xi}\to\mathbb C$. Since $F_n$ converges to
$F$, possibly passing to a subsequence, there exists a negligible set
$E$ such that  
 \[
   \lim_{n\to +\infty} F_n(\xi) = \lim_{n\to
     +\infty}\widetilde{F}_n(j(\xi))= F(\xi),\qquad  \xi\notin E.
 \]
 Define $\widetilde{F}\colon \widetilde{\Xi}\to\mathbb C $ as
 \[
  \widetilde{F}(\widetilde{\xi} ) = \begin{cases}
     \displaystyle{\lim_{n\to +\infty}\widetilde{F}_n(\widetilde{\xi}) }&
     \widetilde{\xi}\in j(\Xi\setminus E), \\
     0 & \widetilde{\xi}\notin  j(\Xi\setminus E). 
   \end{cases}
   \]
Then by construction
\[
F(\xi)= \widetilde{F}(j(\xi)),\qquad \xi\notin E.
  \]
It follows that $L^2(\Xi,\D\xi)_0$ is closed. We now prove that it is
$\hat{\pi}$-invariant.   Given $g\in G$, for all $F\in L^2(\Xi,\D\xi)_0$
\begin{align*}
  \hat{\pi}(g)F(\xi) & = \beta(g)^{-1/2} F(g^{-1}.\xi)\\&= \beta(g)^{-1/2}
                      \widetilde{F}(j(g^{-1}.\xi)) \\ &
                                                       = \beta(g)^{-1/2}
                                                        \widetilde{F}(g^{-1}. j(\xi)), 
\end{align*}
so that $L^2(\Xi,\D\xi)_0$ is $\hat{\pi}$-invariant.

Assume that $\hat{\pi}$ is irreducible, then 
$L^2(\Xi,\D\xi)_0$ is zero or the full space. Since $\mathcal A$ is not
trivial, there exists a non-zero $f\in\mathcal A$ such 
that $\mathcal R f \in L^2(\Xi,\D\xi)_0$ by~\eqref{GAeq:1}. Furthermore, 
$\mathcal R f\neq 0$ since  $\mathcal I\mathcal R f = \mathcal Q f\neq 0$ because $\mathcal Q$ is
an isometry. Hence $L^2(\Xi,\D\xi)_0$ is non-trivial and
$L^2(\Xi,\D\xi)_0=L^2(\Xi,\D\xi)$. 
\end{proof}

The proof of Theorem~\ref{GAthm:main} will be an immediate consequence of Proposition~\ref{GAnovantesimo} and of the following result.
\begin{proposition}\label{GAirrhat}
  Assume that  $\widetilde{H}\neq H$, then $\hat{\pi}$ is not
irreducible. 
\end{proposition}
\begin{proof}
Suppose by  contradiction that $\hat{\pi}$ is irreducible,  then
$L^2(\Xi,\D\xi)_0= L^2(\Xi,\D\xi)$ by Lemma~\ref{GAlem5}.

We first prove that $\omega$ is a finite measure. 
Fix $f\in L^1(\Xi,\D\xi)\cap L^2(\Xi,{\rm d}\xi)$ such that
$f$ is positive and $f\neq 0$, then there exists
$F\colon \widetilde{\Xi}\to\mathbb C $ such that $F(j(\xi))=f(\xi)$ for all $\xi
\notin E$ where $E\subset\Xi$ is negligible. Hence, by~\eqref{GAeq:6b}
\begin{align*}
0<  \int\limits_{\Xi} f(\xi) \D \xi   
  & =  \int\limits_{\Xi} F(j(\xi)) \D \xi  = 
    C \int\limits_{\widetilde{\Xi}} \left(\,\,
    \int\limits_{j(\xi)=\widetilde{\xi}} 
   F(j(\xi))\D\nu_{\widetilde{\xi}}(\xi) \right) 
    d\widetilde{\xi} \\
  & = C \int\limits_{\widetilde{\Xi}} F(\widetilde{\xi}) \ 
    \nu_{\widetilde{\xi}} (\Xi) \,\D\widetilde{\xi} <\infty.
\end{align*}
It follows that  there exists a negligible subset $\widetilde{E}\subset\widetilde{\Xi}$
such that for all $\widetilde{\xi}\notin\widetilde{E}$, $\nu_{\widetilde{\xi}}(\Xi)$
is finite. By construction of $\nu_{\widetilde{\xi}}$ as image measure
\[
\nu_{\widetilde{\xi}}(\Xi) =\omega(\widetilde{H}/H) <+\infty,\qquad
 \widetilde{\xi}\in\widetilde{\Xi}.
  \] 
  
  \begin{figure}\label{GAfig}
\centering

\begin{tikzpicture}[scale=1.1,domain=-5:5,dot/.style={fill,circle,inner sep=0pt,minimum size=2pt}]  
    
\node at (3.6,7) {$\Xi$};

\draw[-] plot [smooth,tension=1] coordinates{(-4,3.4) (-2,2.8) (0,3.4) (2,3.5) (4,4)};
\draw[-] plot [smooth,tension=1] coordinates{(-4,4) (-2,3.4) (0,4) (2,4.1) (4,4.6)};

\draw[-] plot [smooth,tension=1] coordinates{(-4,5.5) (-2,5.8) (0,6.2) (2,5.8) (4,6.1)};
\draw[-] plot [smooth,tension=1,dash pattern=on -4 off 4] coordinates{(-4,4.9) (-2,5.2) (0,5.6) (2,5.2) (4,5.5)};

\draw (-4,2) -- (4,2);
\node at (3.7,1.8) {$\tilde\Xi$};

\draw (0,1.8) -- (0,7);
\node at (0,1.55) {$\tilde\xi$};
\node[font=\small] at (.5,7) {$j^{-1}(\tilde\xi)$};

\draw[line width=0.3mm] (-3,2) -- (3,2);
\node at (2.7,1.8) {$K$};

\draw[line width=0.25mm] (3,3.7) -- (3,4.3);
\draw[line width=0.25mm] (-3,2.95) -- (-3,3.55);
\node at (2.7,3.9) {$Z$};

\draw [dashed] (3,1.8) -- (3,7);
\draw [dashed] (-3,1.8) -- (-3,7);

\node[dot]  at (0,3.7) {};
\node[font=\scriptsize] at (-0.12,3.65) {$\xi_1$};
\node[font=\scriptsize] at (0.55,3.75) {$V^1_{\tilde\xi}$};
  
\node[dot]  at (0,5.9) {};
\node[font=\scriptsize] at (-0.12,5.85) {$\xi_2$};
\node[font=\scriptsize] at (0.55,5.85) {$V^2_{\tilde\xi}$};

\draw (0,3.7) circle (.29cm);
\draw (0,5.9) circle (.29cm);
\end{tikzpicture}

\caption{The setup considered in the proof of Proposition~\ref{GAirrhat}.}
\end{figure}
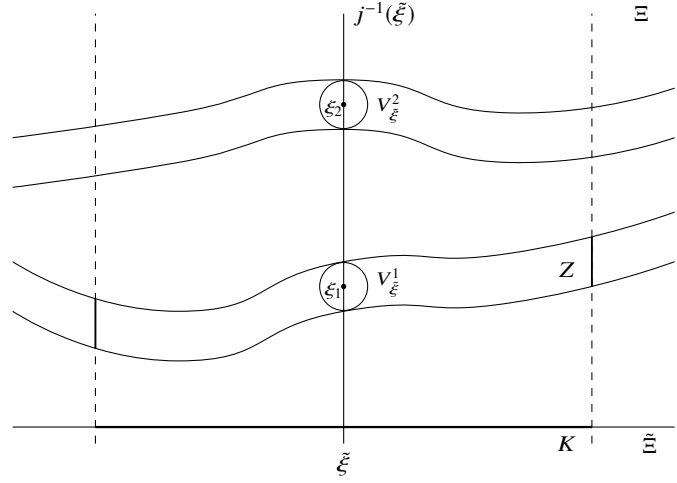
  
For every $\widetilde{\xi}\in\widetilde{\Xi}$,   the map 
$hH\mapsto (\widetilde{\sigma}(\widetilde{\xi})h).\xi_0$ is a
homeomorphism  from $\widetilde{H}/H$ onto  $j^{-1}(\widetilde{\xi})$. Thus, since 
$\widetilde{H}/H$ is not a singleton,  there exist 
$\xi_1,\xi_2\in j^{-1}(\widetilde{\xi})$ such that
$\xi_1\neq \xi_2$ and,  hence, two disjoint compact 
neighbourhoods  $V^1_{\widetilde{\xi}},V^2_{\widetilde{\xi}}$ of $\xi_1$
and $\xi_2$, respectively (see Figure~\ref{GAfig}).  
Since the support of $\nu_{\widetilde{\xi}}$ is $j^{-1}(\widetilde{\xi})$,
then
\begin{equation}\label{GAeq:juve1inter2}
\nu_{\widetilde{\xi}}(V^1_{\widetilde{\xi}})>0,\qquad \nu_{\widetilde{\xi}}(V^2_{\widetilde{\xi}})>0.
\end{equation}

Let now  $K\subset\widetilde{\Xi}$ be a compact
set of positive measure,
\[
Z=\{\xi\in\Xi: j(\xi) \in K,\, \xi\in V^1_{j(\xi)} \} ,
\]
and  $f$ be the corresponding characteristic function of $Z$. By applying~\eqref{GAeq:6b} we obtain
\[
0< \int\limits_{\Xi} |f(\xi)|^2 \D \xi    = \int\limits_{\Xi} f(\xi) \D \xi    = C \int\limits_{K} 
    \nu_{\widetilde{\xi}} ( V^1_{\widetilde{\xi}} ) \D\widetilde{\xi}
    \leq C \omega(\widetilde{H}/H) \int\limits_{K} \D\widetilde{\xi}
    <+\infty,
  \]
which is finite since $\omega(\widetilde{H}/H)<+\infty$ and $K$ is
compact. 
We can apply~\eqref{GAeq:6b} since $f$ is positive  \cite[item c),
Corollary~1, Section~2, No.~8 INT~VII.45]{GAbourbaki}.
Hence $f\in L^2(\Xi,\D\xi)$ and, as above, there exists
$F\colon \widetilde{\Xi}\to\mathbb C $ such that $F(j(\xi))=f(\xi)$ for all $\xi
\notin   E'$ where $E'\subset\Xi$ is negligible. Since $E'$ is negligible, by \eqref{GAeq:6b} applied to  the characteristic function of $E'$
there exists a negligible subset $\widetilde{E'}\subset\widetilde{\Xi}$
such that
\begin{equation}\label{GAeq:juve1toro?}
\nu_{\widetilde{\xi}} \left( E'\cap j^{-1}(\widetilde{\xi})\right) =0,\qquad \widetilde{\xi}\notin \widetilde{E'}.
\end{equation}
Choose an arbitrary $\widetilde\xi\in K\setminus\widetilde
  E'$. By \eqref{GAeq:juve1inter2} and
  \eqref{GAeq:juve1toro?}, for $i=1,2$ there exists $\xi^i\in
  V^i_{\widetilde{\xi}}\setminus E'$ such that
  $j(\xi^i)=\widetilde E'$. Thus 
\[ F(\widetilde{\xi})=F(j(\xi^i))=f(\xi^i)=
  \begin{cases}
    1 & \text{if $i=1$,}\\
    0 & \text{if $i=2$,}
      \end{cases}
\]
which is absurd.
\end{proof}

\section{3D-signals: Radon and ray transforms}\label{GAexamples}
\subsection{The Radon transform on $\mathbb R^3$}\label{GAex:SIM}
\subsubsection{Groups and spaces}
The Radon transform on $\mathbb R^3$ of a signal $f$ is defined as the
integral of $f$ over the set of  planes in $\mathbb R^3$. We
  show that this is an example of our construction.
\par
The input space is $X=\mathbb R^3$ and the
  group is ${\rm SIM}(3)$,  the semi-direct product $\mathbb R^3\rtimes
  K$, with $K=\{aR\in \operatorname{GL}(3,\mathbb R):R\in
  \operatorname{SO}(3),a\in \mathbb R_+\}$. Under the identification 
$K\simeq \operatorname{SO}(3)\times\mathbb R_+$, we write $(\mathbf b,R,a)$
for the elements in ${\rm SIM}(3)$ and the group law becomes  
\[
(\mathbf b,R,a)(\mathbf b',R',a')=(\mathbf b+aR\mathbf b',RR',aa').
\]
A left Haar measure of ${\rm SIM}(3)$ is given by 
\begin{equation}\label{GAHaarmeasure}
{\rm d}\mu(\mathbf b,R,a)=a^{-4}{\rm d}\mathbf b{\rm d}R{\rm d}a,
\end{equation}
where ${\rm d}\mathbf b$ and ${\rm d}a$ are the Lebesgue measures on
$\mathbb R^3$ and $\mathbb R_+$, respectively, and ${\rm d}R$ is a  Haar measure of $\operatorname{SO}(3)$.
The  group ${\rm SIM}(3)$ acts on $\mathbb R^3$ by the canonical action 
\[
(\mathbf b,R,a)[\mathbf x]=\mathbf b+aR\mathbf x,\qquad (\mathbf b,R,a)\in {\rm SIM}(3),\, \mathbf x\in\mathbb R^3.
\]
The isotropy at the origin $\mathbf x_0=0$ is the subgroup $\{(0,k):k\in K\}$ which we identify with $K$, so that $X\simeq {\rm SIM}(3)/K$. Furthermore, the Lebesgue measure ${\rm d}\mathbf x$ on $\mathbb R^3$ is a relatively ${\rm SIM}(3)$-invariant measure with positive character $\alpha(\mathbf 
b,R,a)=a^3$. It remains to choose the closed subgroup $H$ of ${\rm SIM}(3)$
 in such a way that $\{\widehat{\xi}\}$ is the set of planes in $\mathbb R^3$.
We consider $ H = (\mathbb
R^2\times\{0\}) \rtimes  (\operatorname{O}(2) \times \mathbb R_+)$, where
$\operatorname{O}(2)$  
denotes the subgroup of rotations leaving 
  the plane $z=0$   invariant, i.e.\ it consists of the matrices of the form
  \[
    R=
    \begin{bmatrix}
    R_{\pm} R_\theta   & 0\\
  0 & \pm 1 
\end{bmatrix},
\]
where $R_\theta\in {\rm SO}(2)$, $R_+$ is the identity and $R_-=[
\begin{smallmatrix}
 1 & 0 \\ 0 & -1 
\end{smallmatrix}]$. By \eqref{GAeq:21}, the root manifold is
 the $xy$-plane
\[
\hat{\xi_0}=H[\mathbf x_0]=\{\mathbf x\in\mathbb R^3\colon \mathbf x\cdot \mathbf e_3=0\}
\]
and it is easy to verify that $m_0={\rm d}x{\rm d}y$ is a relatively $H$-invariant measure on $\hat{\xi}_0$ with character $\gamma(\mathbf b,R,a)=a^2$.
Furthermore, for each $\xi=(\mathbf b,R,a)H\in\Xi={\rm SIM}(3)/H$, by \eqref{GAeq:31} we compute 
\[
\hat{\xi}=(\mathbf b,R,a)[\hat{\xi}_0]=\{\mathbf x\in\mathbb R^3\colon R\mathbf{e}_3\cdot\mathbf{x}=R\mathbf{e}_3\cdot\mathbf{b}\},
\]
which is the plane perpendicular to the vector $R\mathbf{e}_3$ and passing through $\mathbf b$. It is worth observing that $H$ is the maximal closed subgroup of ${\rm SIM}(3)$ which 
satisfies $h[\hat{\xi}_0]=\hat{\xi_0}$ for every $h\in H$ and then, by Proposition~\ref{GAnovantesimo}, the map $\xi\to\hat\xi$ is injective. Had we chosen $ H' = (\mathbb
R^2\times\{0\}) \rtimes  (\operatorname{SO}(2) \times \mathbb R_+)$, we would have had $H'\subsetneq\widetilde H$ and, indeed, the map $\xi'\mapsto\widehat{\xi'}$ would not have been injective, since every plane would have been labeled by two different $\xi'$.
\par
  We identify the coset space $\Xi={\rm SIM}(3)/H$ with  the  set 
$[0,\pi)_*^2\times\mathbb R$, where
  \[
  [0,\pi)_*^2 = \bigl([0,\pi)\times(0,\pi)\bigr)\cup\{(0,0)\},
  \]
and we write 
$$\mathbf n(\theta,\varphi)={^t(\sin{\varphi}\cos{\theta}, \sin{\varphi}\sin{\theta},\cos{\varphi})}$$
for every $(\theta,\varphi)\in [0,\pi)_*^2$. The group ${\rm SIM}(3)$ acts on $ [0,\pi)_*^2\times\mathbb R$ by the transitive action 
\[
(\mathbf b,R,a).(\theta,\varphi,t)  = (\theta_R,\varphi_R, R\,\mathbf n(\theta,\varphi)\cdot\mathbf n(\theta_R,\varphi_R)(at + R\, \mathbf n(\theta,\varphi) \cdot \mathbf b)),
\]
where $(\theta_R,\varphi_R)\in [0,\pi)_*^2$ is such that $R\,\mathbf n(\theta,\varphi)=\pm \mathbf n(\theta_R,\varphi_R)$. Since the stability subgroup at $(0,0,0)$ is $H$,
  then
${\rm SIM}(3)/H\simeq  [0,\pi)_*^2\times\mathbb R$ under the canonical
isomorphism  
$(\mathbf b,R,a)H\mapsto (\mathbf b,R,a).(0,0,0)$.
\par
We endow $\Xi$ with the measure ${\rm d}\xi=\sin{\varphi}\,{\rm d}\theta{\rm d}\varphi{\rm d}t$, where ${\rm d}\theta$, ${\rm d}\varphi$ and ${\rm d} t$ are the Lebesgue measure on $[0,\pi)$ and $\mathbb{R}$, respectively. 
It is easy to verify that ${\rm d}\xi$ is a relatively ${\rm SIM}(3)$-invariant measure on $\Xi$ with positive character $\beta(\mathbf b,R,a)=a$.

\subsubsection{The representations}
The group ${\rm SIM}(3)$ acts on $L^2(\mathbb R^3)$ by means of
the unitary  representation $\pi$ defined by 
\begin{equation}\label{GArepresentationpi}
\pi(\mathbf b,R,a)f(\mathbf
x)=a^{-\frac{3}{2}}f(a^{-1}R^{-1}(\mathbf{x-b})),
\end{equation}
The dual action $\mathbb R^3\times K\ni(\eta,k)\mapsto {^tk\eta}$ has a single open orbit $\mathcal{O}=\mathbb{R}^3$ for $\mathbf e_3$ of full measure and the stabilizer $K_{\mathbf e_3}\simeq {\rm SO}(2)\times\{1\}$ is compact. Then, the representation 
$\pi$ is irreducible and square-integrable, see \cite{GAantomure96}.

Furthermore, the quasi-regular representation $\hat{\pi}$ of ${\rm SIM}(3)$ acting on $L^2(\Xi,{\rm d}\xi)$ as
\[
\hat{\pi}(\mathbf b,R,a)F(\theta,\varphi,t)=a^{-\frac{1}{2}}F\left(\theta_{R^{-1}},\varphi_{R^{-1}},R^{-1}\mathbf n(\theta,\varphi)\cdot \mathbf n(\theta_{R^{-1}},\varphi_{R^{-1}})\frac{t-\mathbf n(\theta,\varphi)\cdot\mathbf b}{a}\right),
\]
is irreducible, too.  As a consequence, Theorem~\ref{GAthm:main} guarantees that the map $\xi\mapsto\hat\xi$ is injective. Let us consider again the situation with the choice $ H' = (\mathbb
R^2\times\{0\}) \rtimes  (\operatorname{SO}(2) \times \mathbb
R_+)$. In this case, $\Xi'=\rm {SIM}(3)/H'$ may be identified 
with $ S^2\times\mathbb R$, 
and we have already observed that the map
\[
\xi'=(\mathbf n,t)\mapsto \hat\xi_{\mathbf n,t}=\{\mathbf x\in\mathbb
R^3: \mathbf n\cdot \mathbf x=t\}
  \]
is not injective, since $(\mathbf n_1,t_1)$ and $(\mathbf n_2,t_2)$ identify the same plane if
\begin{equation}\label{GAeq:planes}
(\mathbf n_1,t_1)=\pm(\mathbf n_2,t_2).
\end{equation}
In the notation of Section~\ref{GAdpi}, this corresponds to
 $j(\mathbf n_1,t_1)=j(\mathbf n_2,t_2)$. 
According to Theorem~\ref{GAthm:main}, this implies that the corresponding quasi-regular representation $\widehat{\pi'}$ cannot be irreducible. Let us verify this explicitly, in order to visualise the link between the irreducibility of $\widehat{\pi'}$ and the injectivity of $\xi'\mapsto\widehat{\xi'}$ in this example. By arguing as above, it is easy to prove that 
\[
\widehat{\pi'}(\mathbf b,R,a)F(\mathbf n,t)=a^{-\frac{1}{2}}F\left(R^{-1}\mathbf n,\frac{t-\mathbf n\cdot\mathbf b}{a}\right).
\]
Thus, using the notation of Lemma~\ref{GAlem5}, the set
\[
L^2(\Xi',d\xi')_0=\{F\in L^2(\Xi',d\xi'): F(\mathbf n_1,t_1)=F(\mathbf n_2,t_2) \text{ if \eqref{GAeq:planes} holds}\}
\]
is a closed $\widehat{\pi'}$-invariant proper subspace of $L^2(\Xi',d\xi')$. Hence, 
$\widehat{\pi'}$ is not irreducible.

\subsubsection{The Radon transform}

In order to define the Radon transform we need to endow each
$\hat{\xi}$ with a suitable measure. Since the measure $m_0$ is
$H$-relatively invariant, the choice of the representative of $\xi$ is
crucial. We fix the  Borel section
$$\sigma\colon \Xi\to G,\qquad \sigma(\theta,\varphi,t)=(t\mathbf n(\theta,\varphi), R_{\theta,\varphi}, 1),$$
with $R_{\theta,\varphi}\in {\rm SO}(3)$ such that $R_{\theta,\varphi}\mathbf e_3=\mathbf n(\theta,\varphi)$.
We observe that, since $\gamma$ extends to a positive character of $G$, assumption (A4) is implied by the stronger condition $\gamma(\sigma(\xi))=1$ for every $\xi\in\Xi$.
Then, we compute the Radon transform by \eqref{GAeq:radondualpairs} obtaining 
\begin{align}\label{GAeq:radontransformplanes}
\mathcal R f(\theta,\varphi,t)
&=\int\limits_{\mathbb R^2}  f\bigl(t\mathbf n(\theta,\varphi)+R_{\theta,\varphi}(x,y,0)\bigr)  {\rm d}x {\rm d}y,
\end{align}
which is the integral of $f$ on the plane of equation $\mathbf{n}(\theta,\varphi)\cdot\mathbf{x}=t$.  As a consequence of Fubini theorem, equation~\eqref{GAeq:radontransformplanes} makes sense for instance if $f\in L^1(\mathbb R^3)$. 

We recall a crucial result in Radon transform theory in its standard version, known as Fourier slice theorem. We denote by $\operatorname{I}$ the identity operator.
  \begin{proposition}\label{GA:fst}
For every  $f\in L^1(\mathbb R^3)$ 
\begin{equation}
(I\otimes\mathcal{F})\mathcal{R}f(\theta,\varphi,\tau)=\mathcal{F}f(\tau \mathbf n(\theta,\varphi)),
\end{equation}
for all $(\theta,\varphi,\tau)\in [0,\pi)_*^2\times\mathbb R$.
\end{proposition}  
Here the Fourier transform on the right-hand side is in $\mathbb R^3$, whereas the
operator $\mathcal F$ on the left-hand side is one-dimensional and acts on the
variabile~$t$. We repeat this slight abuse of notation in other
formulas below.

We show that assumption~\ref{GAass:A} holds true.  Let $S^2$ be the sphere in $\mathbb R^3$ and denote by $\mathcal{S}(\mathbb R^3)$ and $\mathcal{S}( S^2\times \mathbb R )$ the Schwartz spaces of rapidly decreasing functions on $\mathbb R^3$ and on $ S^2\times \mathbb R $, respectively, and by $\mathcal{S}'(\mathbb R^3)$ and  $\mathcal{S}'(S^2\times \mathbb R)$  the corresponding spaces of tempered distributions; see  \cite[Chapter 1.2]{GAhelgason99} for the definition on $S^2\times \mathbb R$.

We extend the Radon transform $\mathcal{R}$ as an even function on $S^2$ and we denote it by $\mathcal{R}_{e}$, i.e.
\begin{equation*}
\mathcal{R}_{e}f(\mathbf u,t)=\mathcal{R} f(\theta_{\mathbf u}, \varphi_{\mathbf u}, \mathbf{u}\cdot\mathbf{n}(\theta_{\mathbf u}, \varphi_{\mathbf u}) t),
\end{equation*}
where $(\theta_{\mathbf u}, \varphi_{\mathbf u})\in [0,\pi)_*^2$ is such that $\mathbf{n}(\theta_{\mathbf u}, \varphi_{\mathbf u})=\pm \mathbf{u}$.

We recall that, since $\mathcal R_e$ is a continuous map from $\mathcal{S}(\mathbb R^3)$ into $\mathcal{S}(S^2\times\mathbb{R})$ (see \cite{GAhelgason65}), given $F\in \mathcal{S}'(S^2\times\mathbb{R})$, the tempered distribution $\mathcal R_e^{\#}F\colon\mathcal{S}(\mathbb R^3)\to\mathbb C$ given by
\begin{equation*}
\langle \mathcal R_e^{\#}F, f \rangle=\langle F, \mathcal R_e f\rangle
\end{equation*}
is well-defined. If $F\in \mathcal{S}(S^2\times\mathbb{R})$, by Theorem 1.4 in \cite[Chapter 2]{GAnatterer}, the tempered distribution $\mathcal{F}\mathcal R_e^{\#}F$ is represented by the function 
\begin{equation}\label{FBeqn:freqexpressiondualradon}
\mathcal{F}\mathcal R_e^{\#}F(\mathbf v)=|\mathbf v|^{-2}[(I\otimes\mathcal{F})F(\mathbf v/|\mathbf v|,|\mathbf v|)+(I\otimes\mathcal{F})F(-\mathbf v/|\mathbf v|,-|\mathbf v|)].
\end{equation}
By equation \eqref{FBeqn:freqexpressiondualradon}, $\mathcal R_e^{\#}F$
is in $L^2(\mathbb R^3)$ provided that
\begin{equation}
\int\limits_{\mathbb R}t^m F(\mathbf{u},t)\D t=0,\qquad 
m\in\mathbb{N}. \label{eq:1}
\end{equation}
We fix a non-zero $F\in \mathcal{S}(S^2\times\mathbb{R})$ which
satisfies~\eqref{eq:1} and the symmetry condition $F(\mathbf{u},t)=F(-\mathbf{u},-t)$ and we denote  its restriction to $[0,\pi)_*^2\times\mathbb{R}$ by $F_0$, that is
\begin{equation*}
F_0(\theta,\varphi,t)=F(\mathbf{n}(\theta,\varphi),t),
\end{equation*}
for every $(\theta,\varphi,t)\in [0,\pi)_*^2\times\mathbb{R}$. Then, there exists a positive constant $C$ such that
\begin{equation*}
|\langle F_0, \mathcal R f\rangle_{L^2( [0,\pi)_*^2\times\mathbb{R})}|=\frac{1}{2}|\langle F, \mathcal R_e f\rangle_{L^2(S^2\times\mathbb{R})}|=|\langle \mathcal R_e^{\#}F, f \rangle|\leq C\|f\|,
\end{equation*}
for any $f\in \mathcal{S}(\mathbb R^3)$. Therefore, if we take $f_0\in
\mathcal{S}(\mathbb R^3)$ and define the vector subspace $\mathcal
A=\operatorname{span}\{\pi(g) f_0:g\in G
\}\subseteq\mathcal{S}(\mathbb R^3)$, then the domain of the adjoint
of the restriction of $\mathcal R$ to $\mathcal A$ is non-trivial
since $F_0\in\operatorname{dom}(\mathcal R^*)$ and assumption~\ref{GAass:A} holds true.

\subsubsection{The unitarization theorem}
By Theorem~\ref{ernestoduflomoore}, the Radon transform
    $\mathcal{R}\colon\mathcal A\to L^2(\Xi,{\rm d}\xi)$ admits a
    unique closure $\overline{\mathcal{R}}$ which satisfies 
    \begin{equation}
      \label{character}
      \overline{\mathcal{R}}\pi(\mathbf b, R, a) =\chi(\mathbf b,R,a)^{-1}\hat{\pi}(\mathbf b, R, a) \overline{\mathcal{R}},\qquad (\mathbf b, R, a)\in G, 
    \end{equation}
     where $\chi(\mathbf b,R,a)=a$ since $\alpha(\mathbf b, R, a)=a^3$, $\beta(\mathbf b, R, a)=a$ and $\gamma(\mathbf b, R, a)=a^2$. Furthermore, there exists a unique positive self-adjoint operator
    \[ \mathcal{I}\colon \operatorname{dom}(\mathcal{I}) \supseteq
      \operatorname{Im}\overline{\mathcal{R}}\to L^2(\Xi,{\rm d}\xi),
    \]
    semi-invariant with weight $\chi(\mathbf b, R, a)^{-1}=a^{-1}$ with the property
    that the composite operator $\mathcal I \overline{\mathcal{R}}$ extends
    to a unitary operator
    $\mathcal Q\colon L^2(X,{\rm d}x)\to L^2(\Xi,{\rm d}\xi)$ intertwining
    $\pi$ and $\hat{\pi}$, namely
    \begin{equation}\label{intertwiningU}
      \hat{\pi}(g)\mathcal Q\pi(g)^{-1}=\mathcal Q,
      \qquad g\in G.
    \end{equation}
    We can provide an explicit formula for $\mathcal I$.

Consider the subspace
\[
\mathcal D=\{f\in L^2( [0,\pi)_*^2\times\mathbb R):\int\limits_{ [0,\pi)_*^2\times\mathbb R}|\tau|^2|(I\otimes\mathcal F) f(\theta,\varphi,\tau)|^2\sin{\varphi}\,{\rm d}\theta{\rm d}\varphi{\rm d}\tau<+\infty\}
\]
and define the operator $\mathcal J\colon\mathcal D\to L^2( [0,\pi)_*^2\times\mathbb R)$ by
\begin{equation}\label{eq:operatorJ2}
(I\otimes\mathcal F)\mathcal J f(\theta,\varphi,\tau)=|\tau|(I\otimes\mathcal F) f(\theta,\varphi,\tau),
\end{equation}
a Fourier multiplier with respect to the variable $t$. A direct calculation shows that
$\mathcal J$ is a densely defined positive self-adjoint 
injective operator  and is semi-invariant with weight $\zeta(g)=\chi(g)^{-1}=a^{-1}$. By \cite[Theorem 1]{GAdumo76}, there exists $c>0$ such that $\mathcal I= c
\mathcal J$ and we now show that $c=1$.
Take a non-zero function $f\in \mathcal A$. Then, by Plancherel theorem
and Proposition~\ref{GA:fst} we have that 
\begin{align*}
\|f\|^2=\|\mathcal I\mathcal{R}f\|^2_{L^2( [0,\pi)_*^2\times\mathbb R)}&=c^2\|(I\otimes\mathcal F)\mathcal J\mathcal{R}f\|^2_{L^2( [0,\pi)_*^2\times\mathbb R)}\\
&=c^2\, \int\limits_{[0,\pi)_*^2\times\mathbb R}|(I\otimes\mathcal F)\mathcal{R}f(\theta,\varphi,\tau)|^2|\tau|^2\sin{\varphi}\,\D\theta{\rm d}\varphi\D\tau\\
&=c^2\, \int\limits_{[0,\pi)_*^2\times\mathbb R}|\mathcal F f(\tau \mathbf n(\theta,\varphi))|^2|\tau|^2\sin{\varphi}\,\D\theta{\rm d}\varphi\D\tau\\
&=c^2\|f\|^2.
\end{align*}
Thus, we obtain $c=1$. 
\subsubsection{The inversion formula}
By Theorem~\ref{GAgeneralinversionformula}, for any $f\in\mathcal A$ we have the reconstruction formula 
\begin{equation*}
f=\int\limits_{\mathrm{SIM}(3)}a^{-\frac{9}{2}}\langle\overline{\mathcal{R}}f,\hat{\pi}(\mathbf b, R, a)\Psi\rangle_{L^2(\Xi,{\rm d}\xi)}\psi(a^{-1}R^{-1}(\mathbf{x-b})){\rm d}\mathbf b{\rm d}R{\rm d}a,
\end{equation*}
where the integral is weakly convergent and where we used that $\chi(\mathbf b,R,a)=a$, the expression of the Haar measure of SIM$(3)$ given in \eqref{GAHaarmeasure} and the expression of $\pi$ given in \eqref{GArepresentationpi}. 

\subsection{The X-ray transform.}\label{GAxrt} 

The X-ray transform in the Euclidean 3-space maps a function on $\mathbb R^3$ into the set of integrals over the lines and the X-ray reconstruction problem consists in reconstructing a signal $f$ by means of its line integrals.

\subsubsection{Groups and spaces}
Take the same group $G={\rm SIM}(3)$ as in subsection~\ref{GAex:SIM},
namely $G=\mathbb R^3\rtimes K$, with $K=\{aR\in \operatorname{GL}(3,\mathbb R):R\in \operatorname{SO}(3),a\in \mathbb R_+\}$. Firstly, we choose $X=\mathbb R^3$ and, for what concerns this space, we keep the notation as in subsection~\ref{GAex:SIM}.
Then, we consider the space $\Xi=G/H$, where $ H = (\{(0,0)\}\times\mathbb R) \rtimes  (\times \mathbb R_+)$. By \eqref{GAeq:21}, the root manifold is then 
\[
 \hat{\xi_0} =\{ t
                \mathbf e_3\colon t\in \mathbb R\}
\]
and it is easy to verify that $m_0={\rm d}t$ is a relatively $H$-invariant measure on $\hat{\xi}_0$ with character $\gamma(\mathbf b,R,a)=a$.
Furthermore, for each $\xi=(\mathbf b,R,a)H\in\Xi$, by \eqref{GAeq:31} we compute 
\[
\hat{\xi}=(\mathbf b,R,a)[\hat{\xi}_0]=\{tR\mathbf e_3+\mathbf b\colon t\in \mathbb R\},
\]
which is the line parallel to the vector $R\mathbf{e}_3$ and passing through the point $\mathbf b$. It is worth observing that $H$ is the maximal closed subgroup of ${\rm SIM}(3)$ which 
satisfies $h[\hat{\xi}_0]=\hat{\xi_0}$ for every $h\in H$ and then, by Proposition~\ref{GAnovantesimo}, the map $\xi\to\hat\xi$ is injective. 

The coset space $\Xi={\rm SIM}(3)/H$ can be identified with the set $T=\{(\theta,\varphi,\mathbf t):(\theta,\varphi)\in[0,\pi)_*^2,\, \mathbf t\in (\theta,\varphi)^{\bot}\}$, where $(\theta,\varphi)^{\bot}$ denotes the plane passing through the origin and perpendicular to the vector $\mathbf n(\theta,\varphi)$, i.e. the plane of equation $\mathbf n(\theta,\varphi)\cdot\mathbf x=0$. The group ${\rm SIM}(3)$ acts on $T$ by the action
\begin{equation*}
(\mathbf b,R,a).(\theta,\varphi,\mathbf t)  = (\theta_R,\varphi_R,\mathbf t + aR \mathbf b - (\mathbf n(\theta_R,\varphi_R)\cdot(\mathbf t + aR \mathbf b))\mathbf n(\theta_R,\varphi_R)),
\end{equation*} 
where we recall that $(\theta_R,\varphi_R)\in [0,\pi)_*^2$ is such that $R\,\mathbf n(\theta,\varphi)=\pm\mathbf n(\theta_R,\varphi_R)$. Since the stability subgroup at $(0,0,0)$ is $H$,
  then
${\rm SIM}(3)/H\simeq T$ under the canonical
isomorphism  
$(\mathbf b,R,a)H\mapsto (\mathbf b,R,a).(0,0,0)$.

We endow $\Xi$ with the measure ${\rm
  d}\xi=\sin{\varphi}\,\D\theta{\rm d}\varphi{\rm d}\mathbf t$, with
${\rm d}\theta$, ${\rm d}\varphi$ and ${\rm d} \mathbf t$ 
  being the Lebesgue measure on $[0,\pi)$ and $\mathbb{R}^3$, respectively. It is easy to verify that ${\rm d}\xi$ is a relatively ${\rm SIM}(3)$-invariant measure on $\Xi$ with positive character $\beta(\mathbf b,R,a)=a^3$.
    
\subsubsection{The representations}
We recall that the group ${\rm SIM}(3)$ acts on $L^2(\mathbb R^3)$ by means of
the unitary irreducible representation $\pi$ defined by 
$$\pi(\mathbf b,R,a)f(\mathbf x)=a^{-\frac{3}{2}}f(a^{-1}R^{-1}(\mathbf{x-b})).$$

Furthermore, the quasi-regular representation $\hat{\pi}$ of ${\rm SIM}(3)$ acting on $L^2(\Xi,{\rm d}\xi)$ as
\begin{align*}&\hat{\pi}(\mathbf b,R,a)F(\theta,\varphi,\mathbf t)=\\
&a^{-\frac{3}{2}}F\left(\theta_{R^{-1}},\varphi_{R^{-1}},a^{-1}R^{-1}(\mathbf t-\mathbf b)-(\mathbf n(\theta_{R^-1},\varphi_{R^-1})\cdot a^{-1}R^{-1}(\mathbf t-\mathbf b))\mathbf n(\theta_{R^-1},\varphi_{R^-1})\right),
\end{align*}
is irreducible, too.
\subsubsection{The Radon transform}

We fix a Borel section 
$$\sigma\colon \Xi\to G,\qquad \sigma(\theta,\varphi, \mathbf t)=(\mathbf t, R_{\theta,\varphi}, 1),$$
with $R_{\theta,\varphi}\in SO(3)$ such that $R_{\theta,\varphi}\mathbf e_3=\mathbf n(\theta,\varphi)$.
We observe that, since $\gamma$ extends to a positive character of $G$, assumption (A4) is implied by the stronger condition $\gamma(\sigma(\xi))=1$ for every $\xi\in\Xi$.
Then, we compute by \eqref{GAeq:radondualpairs} the Radon transform between the ${\rm SIM}(3)$-transitive spaces $X$ and $\Xi$ obtaining 
\begin{align}\label{GAeq:rxraytransformplanes}
\mathcal R f(\theta,\varphi,\mathbf t)&
=\int\limits_{\mathbb R}  f(t\mathbf n(\theta,\varphi)+\mathbf t)  {\rm d}t,
\end{align}
which is the integral of $f$ over the line parallel to the vector $\mathbf n(\theta,\varphi)$ and passing through the point $\mathbf t\in\mathbb{R}^3$.  
Let us now determine a suitable $\pi$-invariant subspace $\mathcal A$ of $L^2(\mathbb R^3)$ as in (A7). In order to do that, it is useful to derive a Fourier slice theorem for $\mathcal R$.

 For any 
$f\in\mathcal{S}(\mathbb R^3)$, by Theorem 1.1 in \cite[Chapter 2]{GAnatterer}, we have 
\begin{equation}\label{GAeq:fourierslicetheoremxray}
(I\otimes\mathcal{F})\mathcal R f(\theta,\varphi,\mathbf v)=\mathcal{F}f(\mathbf v),\qquad \mathbf v\in(\theta,\varphi)^{\bot}.
\end{equation}
As a consequence, by Plancherel theorem and formula (2.8) in \cite[Chapter 7]{GAnatterer}, we obtain 
\begin{align*}
\|\mathcal{R}f\|_{L^2(\Xi)}^2&=\int_{0}^\pi\int_0^\pi\int\limits_{(\theta,\varphi)^{\bot}}|(I\otimes\mathcal{F})\mathcal R f(\theta,\varphi,\mathbf v)|^2\sin{\varphi}\,{\rm d}\mathbf v\D\theta{\rm d}\varphi\\
&=\int_{0}^\pi\int_0^\pi\int\limits_{(\theta,\varphi)^{\bot}}|\mathcal{F}f(\mathbf v)|^2\sin{\varphi}\,{\rm d}\mathbf v\D\theta{\rm d}\varphi\\
&=\int\limits_{\mathbb R^3}\frac{|\mathcal{F}f(\mathbf v)|^2}{|\mathbf v|}{\rm d}\mathbf v.
\end{align*}
By using spherical coordinates, we obtain
\begin{align*}
\|\mathcal{R}f\|_{L^2(\Xi)}^2&=\int_0^\pi\int_0^\pi\int\limits_{\mathbb R}|\mathcal{F}f(\tau n(\theta,\varphi))|^2|\tau|\sin{\varphi}\,{\rm d}\tau\D\theta{\rm d}\varphi\\
&\leq \int_0^\pi\int_0^\pi\int\limits_{|\tau|\leq1}|\mathcal{F}f(\tau n(\theta,\varphi))|^2\sin{\varphi}\,{\rm d}\tau\D\theta{\rm d}\varphi\\
&+ \int_0^\pi\int_0^\pi\int\limits_{|\tau|>1}|\tau||\mathcal{F}f(\tau n(\theta,\varphi))|^2\sin{\varphi}\,{\rm d}\tau\D\theta{\rm d}\varphi\\
&\leq 4\pi\|f\|_1^2+\|f\|_2^2<+\infty,
\end{align*}
which proves that $\mathcal{R}f\in L^2(\Xi)$ for any $f\in\mathcal{S}(\mathbb R^3)$ and we set $\mathcal{A}=\mathcal{S}(\mathbb R^3)$. Next, we show that $\mathcal{R}$, regarded as an operator from $\mathcal{A}$ to $L^2(\Xi)$ is closable.  By \cite[Theorem VIII.1]{GAreedsimon80}, this is equivalent to proving that the adjoint of $\mathcal{R}f\colon\mathcal{A}\to L^2(\Xi)$ is densely defined.
Suppose that $(f_n)_n\subseteq\mathcal{A}$ is a
sequence such that $f_n\to f$ in $L^2(\mathbb R^3)$ and $\mathcal R f_n\to
g$ in $L^2(\Xi)$. Since $I\otimes\mathcal{F}$ is unitary
from $L^2(\Xi)$ onto $L^2(\Xi)$, we have that $(I\otimes\mathcal{F})\mathcal R f_n\to(I\otimes\mathcal{F})g$ in $L^2(\Xi)$. Since
$f_n\in\mathcal{A}$, by \eqref{GAeq:fourierslicetheoremxray}, for  every
$(\theta,\varphi)\in [0,\pi)_*^2$ 
\begin{align*}
(I\otimes\mathcal{F})\mathcal R f_n(\theta,\varphi,\mathbf v)&=\mathcal{F} f_n(\mathbf v),\quad \mathbf v\in (\theta,\varphi)^{\bot}.
\end{align*}
Hence, passing to a subsequence if necessary,
\begin{align*}
\mathcal{F} f_n(\mathbf v)\to (I\otimes\mathcal{F})g(\theta,\varphi,\mathbf v)
\end{align*}
for almost every $(\theta,\varphi)\in [0,\pi)_*^2$ and $\mathbf v\in(\theta,\varphi)^{\bot}$. Therefore, for almost every $(\theta,\varphi)\in [0,\pi)_*^2$ and $\mathbf v\in(\theta,\varphi)^{\bot}$
\begin{align*}
(I\otimes\mathcal{F})g(\theta,\varphi,\mathbf v)=\lim_{n\to+\infty}\mathcal{F} f_n(\mathbf v)=\mathcal{F} f(\mathbf v),
\end{align*}
where the last equality holds true using a subsequence if necessary. Therefore, if $(h_n)_n\in\mathcal{A}$ is another sequence such that $h_n\to f$ in $L^2(\mathbb R^3)$ and $\mathcal R h_n\to h$ in $L^2(\Xi)$, then, for almost every $(\theta,\varphi)\in [0,\pi)_*^2$ and $\mathbf v\in(\theta,\varphi)^{\bot}$
\begin{align*}
(I\otimes\mathcal{F})h(\theta,\varphi,\mathbf v)=\mathcal{F} f(\mathbf v).
\end{align*}
Therefore,
\[
(I\otimes\mathcal{F})g(\theta,\varphi,\mathbf v)=(I\otimes\mathcal{F})h(\theta,\varphi,\mathbf v)
\]
for almost every $(\theta,\varphi)\in [0,\pi)_*^2$ and
$\mathbf v\in(\theta,\varphi)^{\bot}$. Then
$\lim_{n\to+\infty}\mathcal R f_n=\lim_{n\to+\infty}\mathcal R h_n$,
and $\mathcal R$ is closable. We denote its closure by $\overline{\mathcal R}$.
\subsubsection{The unitarization theorem}
By Theorem~\ref{ernestoduflomoore}, the Radon transform
    $\mathcal{R}\colon\mathcal A\to L^2(\Xi,{\rm d}\xi)$ admits a
    unique closure $\overline{\mathcal{R}}$ which satisfies 
    \begin{equation}
      \label{character}
      \overline{\mathcal{R}}\pi(\mathbf b, R, a) =\chi(\mathbf b,R,a)^{-1}\hat{\pi}(\mathbf b, R, a) \overline{\mathcal{R}},\qquad (\mathbf b, R, a)\in G, 
    \end{equation}
     where $\chi(\mathbf b,R,a)=a$ since $\alpha(\mathbf b, R, a)=a^3$, $\beta(\mathbf b, R, a)=a$ and $\gamma(\mathbf b, R, a)=a^2$. Furthermore, there exists a unique positive self-adjoint operator
    \[ \mathcal{I}\colon \operatorname{dom}(\mathcal{I}) \supseteq
      \operatorname{Im}\overline{\mathcal{R}}\to L^2(\Xi,{\rm d}\xi),
    \]
    semi-invariant with weight $\chi(\mathbf b, R, a)^{-1}=a^{-1}$ with the property
    that the composite operator $\mathcal I \overline{\mathcal{R}}$ extends
    to a unitary operator
    $\mathcal Q\colon L^2(X,{\rm d}x)\to L^2(\Xi,{\rm d}\xi)$ intertwining
    $\pi$ and $\hat{\pi}$, namely
    \begin{equation}\label{intertwiningU}
      \hat{\pi}(g)\mathcal Q\pi(g)^{-1}=\mathcal Q,
      \qquad g\in G.
    \end{equation}
    We can provide an explicit formula for $\mathcal I$. Consider the subspace
\[
\mathcal D=\{f\in L^2(\Xi):\int\limits_{ [0,\pi)_*^2\times(\theta,\varphi)^{\bot}}|\tau||(I\otimes\mathcal F) f(\theta,\varphi,\tau)|^2\sin{\varphi}\,{\rm d}\theta{\rm d}\varphi{\rm d}\tau<+\infty\}
\]
and define the operator $\mathcal J\colon\mathcal D\to L^2(\Xi)$ by
\begin{equation}\label{eq:operatorJ2}
(I\otimes\mathcal F)\mathcal J f(\theta,\varphi,\tau)=|\tau|^\frac{1}{2}(I\otimes\mathcal F) f(\theta,\varphi,\tau),
\end{equation}
a Fourier multiplier with respect to the variable $t$. A direct calculation shows that
$\mathcal J$ is a densely defined positive self-adjoint 
injective operator  and is semi-invariant with weight $\zeta(g)=\chi(g)^{-1}=a^{-\frac{1}{2}}$. By \cite[Theorem 1]{GAdumo76}, there exists $c>0$ such that $\mathcal I= c
\mathcal J$ and we now show that $c=1$.
Consider a non-zero function $f\in \mathcal A$. Then, by Plancherel theorem, equation \eqref{GAeq:fourierslicetheoremxray} and formula (2.8) in \cite[Chapter 7]{GAnatterer}, we obtain 
\begin{align*}
\|f\|^2=\|\mathcal I\mathcal{R}f\|^2_{L^2(\Xi)}&=c^2\|(I\otimes\mathcal F)\mathcal J\mathcal{R}f\|^2_{L^2(\Xi)}\\
&=c^2\, \int\limits_{ [0,\pi)_*^2\times(\theta,\varphi)^{\bot}}|(I\otimes\mathcal F)\mathcal{R}f(\theta,\varphi,\tau)|^2|\tau|\sin{\varphi}\,\D\theta{\rm d}\varphi\D\tau\\
&=c^2\, \int\limits_{ [0,\pi)_*^2\times(\theta,\varphi)^{\bot}}|\mathcal F f(\tau)|^2|\tau|\sin{\varphi}\,\D\theta{\rm d}\varphi\D\tau\\
&=c^2\|f\|^2.
\end{align*}
Thus, $c=1$ and this concludes the proof.

\subsubsection{The inversion formula}
By Theorem~\ref{GAgeneralinversionformula}, for any $f\in\mathcal{S}(\mathbb{R}^3)$, taking into account equations \eqref{GAHaarmeasure} and \eqref{GArepresentationpi} and that $\chi(\mathbf b,R,a)=a^{\frac{1}{2}}$, the reconstruction formula \eqref{GAinversionformula} reads
\begin{equation*}
f=\int\limits_{\mathrm{SIM}(3)}a^{-5}\langle\overline{\mathcal{R}}f,\hat{\pi}(\mathbf b, R, a)\Psi\rangle_{L^2(\Xi,{\rm d}\xi)}\psi(a^{-1}R^{-1}(\mathbf{x-b})){\rm d}\mathbf b{\rm d}R{\rm d}a,
\end{equation*}
where the integral is weakly convergent.

\begin{acknowledgement}
G.S. Alberti, F. De Mari and E. De Vito are members of the Gruppo Nazionale per
l'Analisi Matematica, la Probabilit\`a e le loro Applicazioni (GNAMPA)
of the Istituto Nazionale di Alta Matematica (INdAM) and together with F. Bartolucci are part of the Machine Learning Genoa Center (MaLGa).
\end{acknowledgement}

\end{document}